\documentclass[12pt]{amsart}
\usepackage{enumerate}
\usepackage{todonotes}
\usepackage[top=32mm, bottom=32mm, left=32mm, right=32mm]{geometry} 

\usepackage{amsmath,amssymb,amsbsy,amsfonts,amsthm,latexsym,
                     amsopn,amstext,amsxtra,euscript,amscd,mathrsfs}
\usepackage{amsmath,amssymb,amsbsy,amsfonts,latexsym,amsopn,amstext,cite,
                                               amsxtra,euscript,amscd,bm}
\usepackage{mathtools}
\usepackage{todonotes}
\usepackage{url}
\usepackage[colorlinks,linkcolor=blue,anchorcolor=blue,citecolor=blue,backref=page]{hyperref}

\renewcommand*{\backref}[1]{}
\renewcommand*{\backrefalt}[4]{%
    \ifcase #1 (Not cited.)%
    \or        %
    \else      %
    \fi}
\renewcommand{\UrlBreaks}{\UrlOrds}

\newtheorem{Th}{Theorem}
\newtheorem{Prop}{Proposition}[section]
\newtheorem{lem}{Lemma}

\newtheorem{rem}{Remark}
\newtheorem{Def}{Definition}

\newtheorem{que}{Question}
\newtheorem*{thank}{\ \ \ Acknowledgment}

\def\scalar(#1,#2){(#1\mid#2)}

\newcommand{\ca}{\mathcal{A}}
\newcommand{\cb}{\mathcal{B}}

\newcommand{\cp}{\mathcal{P}}

\newcommand{\bmu}{\bm \mu}

\newcommand{\R}{{\mathbb{R}}}
\newcommand{\Pro}{{\mathbb{P}}}

\newcommand{\C}{{\mathbb{C}}}
\newcommand{\Z}{{\mathbb{Z}}}
\newcommand{\N}{{\mathbb{N}}}
\newcommand{\E}{{\mathbb{E}}}

\def\pgdc{\textrm{gcd}}

\newcommand{\lamob}{\boldsymbol{\lambda}}
\newcommand{\bnu}{\boldsymbol{\nu}}
\newcommand{\vMan}{\boldsymbol{\Lambda}}
\newcommand{\tend}[3][]{\xrightarrow[#2\to#3]{#1}}

\newcommand{\ds}{\displaystyle}

\newcommand*{\Resize}[2]{\resizebox{#1}{!}{$#2$}}
\newcommand*{\Sup}{\Resize{0.7cm}{sup}}
\newcommand{\TODOH}[1]{{#1}}

\title[On the pointwise convergence of the cubic $\dots$ ]{On the pointwise convergence of the cubic average with multiplicative or von Mangoldt weights}
\author{E. H. El Abdalaoui}
\address{University
of Rouen Normandy, Department of Mathematics, LMRS, UMR 60 85, Avenue de l'Universit\'e, BP.12, 76801
Saint Etienne du Rouvray - France}
\email{elhoucein.elabdalaoui@univ-rouen.fr }
\author{Xiangdong Ye}
\address{Wu Wen-Tsun Key Laboratory of Mathematics, USTC, Chinese Academy of Sciences, Department of Mathematics, University
of science and technology of China, Hefei, Anhui, 230026- China}
\thanks{The second author is supported by NNSF of China  (11431012).\\
}

\email{yexd@ustc.edu.cn}
\keywords{A nonconventional ergodic theorem along cube, nonconventional averages, Ces\`{a}ro mean,
moving average, M\"{o}bius, Liouville and von Mangoldt functions.}
\subjclass[2010]{28D15 (Primary), 05D10, 11B37, 37A45 (Secondary)}

\begin{document}

\date{\today}

{\renewcommand\abstractname{Abstract}
\begin{abstract}
 It is shown that the cubic nonconventional ergodic averages of any order with a bounded aperiodic multiplicative function or von Mangoldt weights converge almost surely.
\end{abstract}

\maketitle

\section{Introduction.}
The purpose of this note is motivated by the recent great interest on the M\"{o}bius
function from the dynamical point view, and by the problem of the multiple recurrence which
goes back to the seminal work of Furstenberg \cite{Fur}. This later problem has nowadays a long history.\\

The dynamical study of M\"{o}bius function was initiated recently by Sarnak in \cite{Sarnak4}.
Therein, Sarnak made a conjecture that  the M\"{o}bius
function is orthogonal to any deterministic dynamical sequence. Notice that this precise the definition of a reasonable
sequence in the  M\"{o}bius randomness law mentioned by Iwaniec-Kowalski in \cite[p.338]{Iwaniec}.  Sarnak further
mentioned that Bourgain's approach can be used to prove that for almost all point $x$ in any measurable dynamical system
$(X,\ca,T,\Pro)$, the M\"{o}bius function is orthogonal to any dynamical sequence $f(T^nx)$, where $f$ is a square integrable function.
For simple proofs and other related results, see  \cite{al-lem} and \cite{Cuny-W}. 
 

Here, we are interested in the pointwise convergence of cubic nonconventional ergodic averages with bounded aperiodic multiplicative functions weight and
von Mangoldt weight.\\

The convergence of cubic nonconventional ergodic averages was initiated by  Bergelson in \cite{Berg},
where convergence in $L^2$ was shown for order 2 and under the extra assumption that all the
transformations are equal. Under the same assumption, Bergelson's result was extended by
Host and Kra for cubic averages of order 3  in \cite{Host-K1}, and for arbitrary order in \cite{Host-K2}.
Assani proved that pointwise convergence of cubic nonconventional ergodic averages of order 3 holds for not necessarily
commuting maps in \cite{Assani1}, and further established the pointwise convergence for cubic averages of
arbitrary order when all the maps are equal.
In \cite{chu-nikos}, Chu and Frantzikinakis completed the study and established the pointwise convergence for
the cubic averages of arbitrary order.
Very recently, Huang-Shao and Ye \cite{hsy-1} gave a topological-like proof of the pointwise convergence
of the cubic nonconventional ergodic average when all the maps are equal. They further applied their
method to obtain the pointwise convergence of nonconventional ergodic averages for a distal system.\\

Here, we establish that the cubic averages of any order with the aperiodic bounded multiplicative function
weight converge to zero almost surely.  The proof depends
heavily on 
the Gowers inverse theorem. As a consequence,
we obtain that the cubic averages of any order with M\"{o}bius or Liouville weights converge to zero
almost surely. Moreover, we establish that  the cubic nonconventional ergodic averages weighted with
von-Mangoldt function converge. This result is obtained as a consequence of
the very recent result of Ford-Green-Konyagin-Tao \cite{FGKtao} combined with the Gowers inverse theorem and the seive methods trick called $W$-trick.
\vskip 0.2cm

The paper is organized as follows. In Section 2, we state
our main results and we recall the main ingredients needed for the proofs. In Section 3, we prove our fist main result. In Section 4, we give a proof of
our second main result.

\section{Basic definitions and tools.}\label{Sec:dt}
In this section we will recall some Basic definitions and tools we will use in this paper, and state our main results.

\subsection{The multiplicative or von Mangoldt functions}
Recall that the Liouville function $\lamob ~~:~~ \N^*\longrightarrow  \{-1,1\}$ is defined by
$$
\lamob(n)=(-1)^{\Omega(n)},
$$
where $\Omega(n)$ is the number of prime factors of $n$ counted with multiplicities with $\Omega(1)=1$. \TODOH{Obviously,
$\lamob$ is completely mutiplicative, that is, $\lamob(nm)=\lamob(n)\lamob(m),$ for any $n,m \in \N^*$}.
The integer $n$ is said to be not square-free if there is a prime number $p$ such that $n$ is in the class of $0$ mod $p^2$.
The M\"{o}bius function
$\bmu ~:~ \N \longrightarrow \{-1,0,1\}$ is define as follows
$$
\bmu(n)=\begin{cases}
\lamob(n),& \text{ if $n$ is square-free ;}\\
1,& \text{ if }n=1;\\
0,& \text{ otherwise}.
\end{cases}
$$
This definition of M\"{o}bius function establishes that the the restriction of Liouville function and M\"{o}bius function to set of
square free numbers coincident. Nevertheless, the  M\"{o}bius function is only mutiplicative, that is,
$\bmu(mn)=\bmu(m)\bmu(n)$ whenever $n$ and $m$ are coprime.\\

We further remind that the von Mangoldt function and it's cousin $\Lambda'$ are given by
$$
\vMan(n)=\begin{cases}
\log(p),& \text{ if $n=p^\alpha,$ for some prime $p$ and $\alpha \geq 1$ ;}\\
0,& \text{ otherwise}.
\end{cases}
$$
and
$$
\Lambda'(n)=
\begin{cases}
\log(n) &\quad\text{if $n$ is a prime,} \\
0&\quad\text{if not.}
\end{cases}
$$
We will denote as customary by $\pi(N)$ the number of prime less than $N$ and by $\cp$ the set of prime. We remind that for $x>0$ the chebyshev functions are defined by

$$\Psi(x)=\sum_{n \leq x}\Lambda(n) ~~~~~~~\textrm{~~and~~} \upsilon(x)=\sum_{n \leq x}\Lambda'(n).$$

By the Prime Number Theorem  with Reminder (PNTR), it is well-known that $\pi(N)$ is realted to the function $Li$ given by
$$Li(N)=\int_2^{N}\frac1{\log(t)}dt.$$
We further have, by Sleberg's estimation,
$$\pi(N)=\frac{N}{\log(N)}+O\big(\frac{N}{(\log(N))^2}\big),$$
This estimation is equivalent to the following relation
$$\upsilon(N)=N+O\big(\frac{N}{(\log(N))}\big).$$



We say that the multiplicative function $\bnu$ is aperiodic if

$$\frac{1}{N}\sum_{n=1}^{N}\bnu(an+b) \tend{N}{+\infty}0,$$
\noindent for any $(a, b) \in \N^* \times \N$.
By Davenport's theorem \cite{Da} and Batman-Chowla's theorem \cite{Bat-Cho}, the M\"{o}bius and Liouville functions are aperiodic.


We need the following notion of statistical orthogonality.
\begin{Def}Let $(a_n),(b_n)$ be two sequences of complex numbers. The sequences $(a_n)$ and $(b_n)$ are said to be statistically orthogonal (or just orthogonal)
if
$$\Big|\frac{1}{N}\sum_{n=1}^{N}a_n \overline{b_n}\Big|= o\Bigg(\Big|\frac{1}{N}\sum_{n=1}^{N}|a_n|^2\Big|\Big)^{\frac12}
\Big(\Big|\frac{1}{N}\sum_{n=1}^{N}|b_n|^2\Big|\Big)^{\frac12}\Bigg).$$
\end{Def}

\subsection{Cubic averages and related topics.}

Let $(X,\cb,\Pro)$ be a Lebesgue probability space and given three measure preserving transformations $T_1,T_2,T_3$ on $X$.
Let $f_1,f_2,f_3 \in L^{\infty}(X)$.  The cubic nonconventional ergodic averages of order $2$ with weight $A$ are defined by   $$\frac1{N^2}\sum_{n,m=1}A(n)A(m)A(n+m)f_1(T_1^nx) f_2(T_2^nx) f_3(T_3^nx).$$
This nonconventional ergodic average can be seen as a classical one as follows
$$\frac1{N^2}\sum_{n,m=1}\widetilde{f_1}({\widetilde{T_1}}^n(A,x))
{\widetilde{f_2}}({\widetilde{T_2}}^m(A,x)) {\widetilde{f_3}}({\widetilde{T_3}}^{n+m}(A,x)) ,$$
where $\widetilde{f_i}=\pi_0\otimes f_i, \widetilde{T_i}=(S \otimes T_i),~~~i=1,2,3$ and
$\pi_0$ is define by $x=(x_n)\longmapsto x_0$ on the space $Y=\C^{N}$ equipped with some probability measure.\\

\noindent More generaly, let  $k \geq 1$ and put
$$C^*=\{0,1\}^k\setminus\{(0,\cdots,0)\}.$$
Consider the family $(T_e)_{e \in C^*}$  of transformations measure preserving on $X$ and for each $e \in C^*$ let  $f_e$ be in  $L^{\infty}(X)$.
The cubic nonconventional ergodic averages of order $k$ with weight $A$ are given by
\begin{equation}\label{def1}
\dfrac1{N^k}\sum_{{\boldsymbol{n}} \in [1,N]^k}\prod_{{\boldsymbol{e}}
\in C^*}A({\boldsymbol{n}}.{\boldsymbol{e}})f_{{\boldsymbol{e}}}
\big(T_{{\boldsymbol{e}}}^{{\boldsymbol{n}}.{\boldsymbol{e}}}x\big) \tend{N}{+\infty}0,
\end{equation}
where ${\boldsymbol{n}}=(n_1,\cdots,n_k)$, ${\boldsymbol{e}}
=(e_1,\cdots,e_k)$,  ${\boldsymbol{n}}.{\boldsymbol{e}}$ is the usual inner product.\\

\noindent The study of the cubic averages is closely and strongly related to the notion of seminorms introduced in \cite{Gowers} and \cite{Host-K2}. 
They are nowadays called Gowers-Host-Kra's seminorms.\\

\noindent Assume that $T$ is an ergodic measure preserving transformation on $X$. Then,  for any $k \geq 1$,
the Gowers-Host-Kra's seminorms on $L^{\infty}(X)$ are defined inductively as follows
 $$\||f|\|_1=\Big|\int f d\mu\Big|;$$
 $$\||f|\|_{k+1}^{2^{k+1}}=\lim\frac{1}{H}\sum_{l=1}^{H}\||\overline{f}.f\circ T^l|\|_{k}^{2^{k}}.$$

\noindent For each $k\geq 1$, the seminorm $\||.|\|_{k}$ is well defined. For details, we refer the reader to \cite{Host-K2} and \cite{Host-Studia}. Notice that
 the definitions of Gowers-Host-Kra's seminorms can be easily extended to non-ergodic maps as it was mentioned by Chu and
  Frantzikinakis in \cite{chu-nikos}.\\

\noindent The importance of the Gowers-Host-Kra's seminorms in the study of the nonconventional multiple ergodic averages
is due to the existence of a $T$-invariant sub-$\sigma$-algebra $\mathcal{Z}_{k-1}$ of $X$ that satisfies
$$\E(f|\mathcal{Z}_{k-1})=0 \Longleftrightarrow \||f|\|_{k}=0.$$
This was proved by Host and Kra in \cite{Host-K2}. The existence of the factors $\mathcal{Z}_{k}$ was also showed by Ziegler in \cite{Ziegler}.
We further notice that Host and Kra established a connection between the $\mathcal{Z}_{k}$ factors and the nilsystems in \cite{Host-K2}.

\subsection{The main results}

At this point we are able to state our  main results

\begin{Th}\label{main}
The cubic nonconventional ergodic averages of any order with a bounded aperiodic multiplicative function
weight converge almost surely to zero, that is, for any $k \geq 1$, for any  $(f_e)_{e \in C^*} \subset L^{\infty}(X)$,
for almost all $x$, we have
\begin{equation}\label{con}
\dfrac1{N^k}\sum_{{\boldsymbol{n}} \in [1,N]^k}\prod_{{\boldsymbol{e}}
\in C^*}\bnu({\boldsymbol{n}}.{\boldsymbol{e}})f_{{\boldsymbol{e}}}
\big(T_{{\boldsymbol{e}}}^{{\boldsymbol{n}}.{\boldsymbol{e}}}x\big) \tend{N}{+\infty}0,
\end{equation}
where ${\boldsymbol{n}}=(n_1,\cdots,n_k)$, ${\boldsymbol{e}}
=(e_1,\cdots,e_k)$, $C^*=\{0,1\}^k\setminus\{(0,\cdots,0)\}$,
${\boldsymbol{n}}.{\boldsymbol{e}}$ is the usual inner product, and $\bnu$ is the bounded aperiodic multiplicative function.
 \end{Th}
Our second main result can be stated as follows

\begin{Th}\label{main-Mangoldt}The cubic nonconventional ergodic averages of any order with von Mangoldt function
weight converge almost surely.
\end{Th}
Notice that Theorem \ref{main-Mangoldt} is related in some sense to the weak correlation of von Mangoldt function.
Moreover, the study of the correlations of von Mangoldt function is of great importance in number theory, since it is related to the
famous old conjecture of the twin numbers and more generaly to Hard-Littlewood k-tuple conjecture. It is also related to Riemann hypothesis and the Goldbach
conjectures.

\subsection{ Some tools from the theory of Nilsystems and nilsequences.}  
The nilsystems are defined in the setting of homogeneous space \footnote{For a nice account of the theory 
of the homogeneous space we refer the reader to \cite{Dani},\cite[pp.815-919]{B-katok}.}. Let $G$ be a Lie
group, and $\Gamma$ a discrete cocompact subgroup (Lattice, uniform subgroup) of $G$. The homogeneous space is given by
$X=G/\Gamma$ equipped with the Haar measure $h_X$ and the canonical complete $\sigma$-algebra
$\cb_{c}$. The action of $G$ on $X$ is by the left translation, that is, for any $g \in G$,
we have $T_g(x\Gamma)=g.x\Gamma=(gx)\Gamma.$ If further $G$ is a nilpotent Lie group of order
$k$,
$X$ is said to be a $k$-step nilmanifold. For any fixed $g \in G$, the dynamical system $(X,\cb_{c},h_X,T_g)$
is called a $k$-step nilsystem. The basic $k$-step  nilsequences on $X$ are defined by $f(g^nx\Gamma)=(f \circ T_g^n)(x\Gamma)$,
where $f$ is a continuous function of $X$. Thus, $(f(g^nx\Gamma))_{n \in \Z}$ is any element of
$\ell^{\infty}(\Z)$, the space of bounded sequences, equipped with uniform norm
$\ds \|(a_n)\|_{\infty}=\sup_{n \in \Z}|a_n|$. A $k$-step nilsequence, is a uniform limit of basic $k$-step nilsequences.
For more details on the nilsequences we refer the reader to  \cite{Host-K3}
and \cite{BHK}\footnote{The term 'nilsequence' was coined by Bergleson-Host and Kra in 2005 \cite{BHK}.}.\\

Recall that the sequence of subgroups $(G_n)$ of $G$ is a filtration if  $G_1=G,$ $G_{n+1}\subset G_{n},$ and
$[G_n,G_p] \subset G_{n+p},$ where $[G_n,G_p]$ denotes the subgroup of $G$ generated by the commutators $[x,y]=
x~y~x^{-1}y^{-1}$ with $x \in G_n$ and $y \in G_p$.
The lower central filtration is given by $G_1=G$ and $G_{n+1}=[G,G_n]$. It is well know that the lower central filtration allows to construct a Lie algebra $\textrm{gr}(G)$ over the ring $\Z$ of integers. $\textrm{gr}(G)$ is called a
graded Lie algebra associated to $G$ \cite[p.38]{Bourbaki2}. The filtration is said to be of degree or length $l$ if
$G_{l+1}=\{e\},$ where $e$ is the identity of $G$.
We denote by $G^e$ the identity component of $G$. Since $X=G/\Gamma$ is compact, we can assume that $G/G^e$ is finitely generated \cite{Leib}.\\

If $G$ is connected and simply-connected  with Lie algebra $\mathfrak{g}$ \footnote{By Lie's fundamental theorems and up
to isomorphism, $\mathfrak{g}=T_eG$, where $T_eG$ is the tangent space at the identity $e$ \cite[p.34]{Kirillov}.},
then $\exp~~:~~G \longrightarrow \mathfrak{g}$ is a diffeomorphism, where $\exp$ denotes the Lie group exponential map.
We further have, by Mal'cev's criterion, that $\mathfrak{g}$ admits a basis $\mathcal{X}=\{X_1,\cdots,X_m\}$ with rational structure constants \cite{Malcev}, that is,
$$[X_i,X_j]=\sum_{n=1}^{m} c_{ijn}X_n,~~~~~~\textrm{for~~all~~~}  1 \leq i,j \leq k,$$
where the constants $c_{ijn}$ are all rational. \\

Let $\mathcal{X}=\{X_1,\cdots,X_m\}$ be a Mal'cev basis of  $\mathfrak{g}$, then any element $g \in G$ can be uniquely written in the form $g=\exp\Big(t_1X_1+t_2X_2+\cdots+t_m X_m\Big),$ $t_i \in \R$, since the map $\exp$ is a diffeomorphism.
The numbers $(t_1,t_2,\cdots,t_k)$ are called the Mal'cev coordinates of the first kind of $g$. In the same manner,
$g$  can be uniquely written in the form $g=\exp(s_1X_1). \exp(s_2X_2).\cdots.\exp(s_m X_m),$ $s_i \in \R$.
The numbers $(s_1,s_2,\cdots,s_k)$ are called the Mal'cev coordinates of the second kind of $g$. Applying Baker-Campbell-Hausdorff formula, it can be shown that the multiplication law in $G$ can be expressed by a polynomial mapping
$\R^m \times \R^m \longrightarrow \R^m$ \cite[p.55]{Onishchik}, \cite{Green-Tao-Orbit}. This gives that any polynomial sequence $g$ in $G$ can be written as follows
$$g(n)=\gamma_1^{p_1(n)},\cdots,\gamma_m^{p_m(n)},$$
where $\gamma_1 ,\cdots,\gamma_m \in G$, $p_i ~~:~~\N \longrightarrow \N$ are polynomials \cite{Green-Tao-Orbit}. Given $n, h \in \Z$, we put
$$\partial_{h}g(n)=g(n+h)g(n)^{-1}.$$
This can be interpreted as a nilsquencediscrete derivative on $G$. Given a filtration $(G_n)$ on $G$,
 a sequence of polynomial $g(n)$ is said to be adapted to $(G_n)$ if  $\partial_{h_i}\cdots \partial{h_1}g$ takes values in $G_i$ for all positive integers $i$ and for all choices of $h_1, \cdots,h_i \in \Z$. The set of all polynomial sequences adapted to $(G_n)$ is denoted by ${\textrm{poly}}(\Z,(G_n))$.\\

  Furthermore, given a Mal'cev's basis ${\mathcal{X}}$ one can induce a right-invariant metric $d_{\mathcal{X}}$ on $X$ \cite{Green-Tao-Orbit}. We remind that for a real-valued function $\phi$ on $X$, the Lipschitz norm is defined by
$$\|\phi\|_{L}=\|\phi\|_{\infty}+\sup_{x \neq y}\frac{\big|\phi(x)-\phi(y)\big|}{d_{\mathcal{X}}(x,y)}.$$
The set $\mathcal{L}(X,d_{\mathcal{X}})$ of all Lipschitz functions is a normed vector space, and for any  $\phi$ and $\psi$ in $\mathcal{L}(X,d_{\mathcal{X}})$, $\phi \psi \in \mathcal{L}(X,d_{\mathcal{X}})$ and
$\|\phi \psi \|_L \leq \|\phi\|_L \|\psi\|_L$. We thus get, by Stone-Weierstrass theorem, that the subsalgebra $\mathcal{L}(X,d_{\mathcal{X}})$ is dense in the space of continuous functions $C(X)$ equipped with uniform norm $\|\|_{\infty}$. 
It turns out that for a Lipschitz function, the extension from an arbitrary subset is possible without
increasing the Lipschitz norm, thanks to Kirszbraun-Mcshane extension theorem \cite[p.146]{Dudley}.\\ 

In this setting, we remind the following fundamental Green-Tao's theorem on the strong orthogonality of the M\"{o}bius
function to any $m$-step nilsequence, $m \ge 1$.
\begin{Prop}\label{Green-Tao-th}\cite[Theorem 1.1]{Green-Tao}.
Let $G/\Gamma$ be a $m$-step  nilmanifold for some $m\ge 1$.
Let $(G_p)$ be a filtration of $G$ of degree $l \ge 1$. Suppose that $G/\Gamma$ has a $Q$-rational Mal'cev basis $\mathcal{X}$ for some $Q \ge 2$, defining
a metric $d_{\mathcal{X}}$ on $G/\Gamma$. Suppose that $F : G/\Gamma\rightarrow [-1, 1]$ is a Lipschitz function. Then, for any $A>0$, we have the bound,
$$\underset{g \in {\textrm{poly}}(\Z, (G_{p}))}{\Sup}\Big|\frac1{N}\sum_{n=1}^{N}\bmu(n)F(g(n)\Gamma)\Big| \leq C\frac{(1 + ||F||_{L})} {\log^{A} N},$$
where the constant $C$ depends on $m,l,A,Q$, $N \geq 2$.
 \end{Prop}
We further need the following decomposition theorem due to Chu-Frantzikinakis and Host from \cite[Proposition 3.1]{chu-Nikos-H}.  
\begin{Prop}[NSZE-decomposition theorem \cite{chu-Nikos-H}]\label{NSZE} Let $(X,\ca,\mu,T)$ be a dynamical system, $f \in  L^{\infty}(X)$,
and $k \in \N$. Then for every $\varepsilon > 0$, there exist measurable functions $f_{ns} ,f_z ,f_e $, such that
\begin{enumerate}[(a)]
\item $\|f_{\kappa}\|_{\infty} \leq 2 \|f\|_{\infty}$ with $\kappa \in \{ns,z,e\}$.
\item $f = f_{ns} + f_z + f_e$ with  $|\|f_z\|_{k+1} = 0;~~ \|f_e\|_1 <\varepsilon;$ and
\item for $\mu$ almost every $x \in X$, the sequence $(f_{ns}(T^nx))_{n \in \N}$ is a $k$-step nilsequence.
\end{enumerate}
\end{Prop}

\section{Proof of the first main result (Theorem \ref{main}).}\label{sec:ms}

The proof of our first main result (Theorem \ref{main}) for $k \geq 2$ is essentially based on the inverse Gowers norms theorem  due to Green,
Tao and Ziegler \cite{Green-Tao-Z} combined with the recent result of
Host-Frantzikinakis \cite{Nikos-Host3}.
We need also the following result due to T. Tao.
\begin{Prop}[Uniform Discrete inverse theorem for Gowers norms, \cite{Taoblog}]\label{DITGN} Let $N \geq 1$
and $s \geq 1$ be integers, and let $\delta>0$. Suppose
$f~:~ \Z \longrightarrow [-1,1]$ is a function supported on $\{1,\cdots,N\}$ such that
$$\dfrac1{N^{s+2}}\sum_{{(n,\boldsymbol{n})} \in [1,N]^{s+1}}\prod_{{\boldsymbol{e}} \in \{0,1\}^{s+1}}f\big(n+{\boldsymbol{n}}.{\boldsymbol{e}}\big) \geq \delta.$$
Then there exists a filtered nilmanifold $G/\Gamma$ of degree $ \leq s$ and complexity $O_{s,\delta}(1)$, a
polynomial sequence $g~~:~~\Z \longrightarrow G$, and a Lipschitz function $F~~:~~G/\Gamma \longrightarrow \R$ of
Lipschitz constant $O_{s,\delta}(1)$ such that
$$\frac1{N}\sum_{n=1}^{N}f(n) F(g(n)\Gamma) \gg_{s, \delta} 1.$$
\end{Prop}
For the definition of complexity, we refer to \cite{Green-Tao-Z} and for the proof of Proposition \ref{DITGN} we refer to \cite{Taoblog}.
Using this version of the discrete inverse theorem for Gowers norms, T. Tao established the continuous version of the inverse theorem for Gowers norms.
According to this version, we notice that the nilsequence is independent of $N$. However, we warn the reader
that the version of the inverse theorem for Gowers norms in \cite{Green-Tao-Z} allows the nilsequence to depend on $N$.\\


At this point, let us give a proof of our first main result.\\

\noindent {\textbf{Proof of Theorem \ref{main}.}}
By our assumption
$\bnu$ is aperiodic, 
Therefore, by Theorem 2.2 from \cite{Nikos-Host3}, for any nilsequence $(u_n)$, we have
$$\frac{1}{N}\sum_{n=1}^{N}\bnu(n)u_n \tend{N}{+\infty}0.$$

Now, for the case $k=1$, we refer to \cite[Section 3]{al-lem}. Let us assume from now that $k \geq 2$.\\

We proceed by contradiction. Assume that  (\ref{con}) does not hold. Then, there exist $\delta>0$, a functions $f_e,$ $e \in V_k$ and $\mu(A)>0$
such that for each $x\in A$
 \[
\limsup_{N \longrightarrow +\infty}\dfrac1{N^k}\sum_{{\boldsymbol{n}} \in [1,N]^k}\prod_{{\boldsymbol{e}} \in C^*}\bnu({\boldsymbol{n}}.{\boldsymbol{e}})f_{{\boldsymbol{e}}}
\big(T_{{\boldsymbol{e}}}^{{\boldsymbol{n}}.{\boldsymbol{e}}}x\big) \geq \delta,
\]
where $\nu$ is an aperiodic bounded multiplicative function. Whence, by \cite[Proposition 3.2]{chu-nikos}, we have
$$\||\bnu(n) f_e\circ T_e^n(x)|\|_{U^{k+1}} \geq \delta,$$
for some $e \in C^*$ and  for any $x \in A$.\\

This combined with Proposition \ref{DITGN} yields that there exist a nilmanifold $G/\Gamma$, a filtration $(G_p)$, a
polynomial sequence $g$ and a Lipschitz function $F$ such that
$$\frac1{N}\sum_{n=1}^{N}\bnu(n) f_e(T_e^{n}x) F(g(n)\Gamma) \gg_{k, \delta} 1.$$

Applying the decomposition theorem (Proposition \ref{NSZE}), we
way write $f_e=f_{e,ns}+f_{e,z}+f_{e,e}$ and may assume that we have
$|\|f_{e,z}(T_e^n(x))\||_{U^{k+1}}=0$. Notice that we further have

$$\frac1{N}\sum_{n=1}^{N}\bnu(n) f_{e,ns}(T_e^{n}x) F(g(n)\Gamma) \tend{N}{+\infty}0,$$
since the pointwise product of two nilsequences is a nilsequence and $\bnu$ 
is aperiodic combined with the fact that 
the Gowers norms $\big\||\bnu\||_{U^s(N)} \tend{N}{+\infty}0$, $s \geq 2$, by Theorem 2.5 from \cite{Nikos-Host3}.\\

\noindent We thus get, up to some small error, that
 $$\frac1{N}\sum_{n=1}^{N}\bnu(n) f_e(T_e^{n}x) F(g(n)\Gamma)\sim
 \frac1{N}\sum_{n=1}^{N}\bnu(n) f_{e,z}(T_e^{n}x) F(g(n)\Gamma).$$

\noindent Whence, by the spectral theorem \footnote{The spectral measure $\sigma_{f_e}$ is a finite measure on the circle given by
$\ds \widehat{\sigma_{f_e}}(n)=\int f_e(T_e^nx) \overline{f_e(x)}d\mu(x).$}, we can write

\begin{align*}
&\Big\|\frac1{N}\sum_{n=1}^{N}\bnu(n) f_{e,z}(T_e^{n}x) F(g(n)\Gamma)\Big\|_2\\
&=\Big\|\frac1{N}\sum_{n=1}^{N}\bnu(n) z^n F(g(n)\Gamma)\Big\|_{L^2(\sigma_{f_{e,z}})}\gg_{k, \delta} 1.
\end{align*}
Letting $N \longrightarrow +\infty$, we get a contradiction since  $\big\||\bnu\||_{U^s(N)} \tend{N}{+\infty}0$, $s \geq 2$.
The proof of the theorem is complete.

\section{Proof of our second main result (Theorem \ref{main-Mangoldt})}

In the same spirit as in the proof of our first main result we start by proving the following

\begin{Prop}\label{Mangoldt-nilsequence}The cubic nonconventional ergodic averages of any order with von Mangoldt function
weight converge almost surely provided that the systems are nilsystems.
\end{Prop}
The proof of Proposition \ref{Mangoldt-nilsequence} is largely inspired from Ford-Green-Konyagin-Tao's proof of the main theorem in \cite{FGKtao}. It used also some
elementary fact on Gowers uniformity semi-norms.

As it is mentioned in \cite{FGKtao}, the fundamental ingredients in the proof of the main theorem are the M\"{o}bius disjointness of the
nilsequences combined with
the inverse Gowers theorem (Proposition \ref{DITGN}) and the so-called ``W-trick''. 
For the W-trick, we define
$$W=\prod_{\overset{p \in \cp}{p \leq \omega(N)}}p ,$$
where $\omega(N) \leq \frac{1}{2} \log\big(\log(N)\big)$ for large $N \in \N$. Therefore, by the PNT, we have $W=O(\sqrt{\log(N)}).$ For $r<W$ and coprime to $W$, we put
$$\Lambda'_{r,\omega}(n)=\frac{\phi(W)}{W}\Lambda'(Wn+r),~~~ n \in \N.$$
Let $(a_n)$ be a given sequence, then, by the seive methods (see for instance \cite{Iwaniec}), we have
\begin{align*}\label{Seive}
\frac{1}{WN}\sum_{n=1}^{WN}\Lambda'(n)a_n&=\frac{1}{W}\sum_{\overset{r<W}{\pgdc(r,W)=1}}\frac{1}{N}\sum_{n=1}^{N}\Lambda'(Wn+r)a_{Wn+r}+o_W(1)\nonumber \\
&=\frac{1}{\phi(W)}\sum_{\overset{r<W}{\pgdc(r,W)=1}}\frac{1}{N}\sum_{n=1}^{N}\Lambda'_{r,\omega}a_{Wn+r}+o_W(1) \nonumber\\
\end{align*}
Whence
\begin{align}
\frac{1}{WN}\sum_{n=1}^{WN}\Lambda'(n)a_n  \nonumber &=\frac{1}{\phi(W)}\sum_{\overset{r<W}{\pgdc(r,W)=1}}\frac{1}{N}\sum_{n=1}^{N}\big(\Lambda'_{r,\omega}-1\big)a_{Wn+r}\\
& \ \ \ \ +\frac{1}{\phi(W)}\sum_{\overset{r<W}{\pgdc(r,W)=1}}\frac{1}{N}\sum_{n=1}^{N}a_{Wn+r}+o_W(1).
\end{align}
Moreover, if the sequence $(a_n)$ satisfy $a_n=o(n)$ then the convergence of the sequence $\Big(\frac{1}{WN}\sum_{n=1}^{WN}a_n\Big)$ implies the convergence of the sequence
$\Big(\frac{1}{N}\sum_{n=1}^{N}a_n\Big).$\\

We need also the following technical lemma, and we present it in a form which is more precise than we need here, since we hope it may 
find other applications.
\begin{lem}\label{NH}Let $(a_n)$ be a bounded sequence of complex numbers. Then, we have
$$\Big|\frac{1}{\pi(N)}\sum_{\overset{p\textrm{~~prime}}{p \leq N}}a_p-\frac1{N}\sum_{n=1}^{N}\Lambda(n)a_n\Big|
\leq \frac{8C}{\log N}+\frac{6C^2}{(\log N)^2}+\frac{(\log N)^2}{2\sqrt{N}\log 2},$$
where $C$ is some absolute positive constant.
\end{lem}
For the proof of Lemma \ref{NH} we need the following classical result due to Chebyshev \cite[p. 76]{Apostol}

\begin{Prop}\label{chebyshev}For any $x>0$,
$$0 \leq \frac{\Psi(x)-\upsilon(x)}{x} \leq \frac{(\log x)^2}{2\sqrt{x}\log x}.$$
\end{Prop}

\begin{proof}[\textbf{Proof of Lemma \ref{NH}}]  We assume Without lost of generalities that $(a_n)$ is bounded by 1.
 By the triangle inequality, we  have
\begin{eqnarray*}
\Big|\frac{1}{\pi(N)}\sum_{\overset{p\textrm{~~prime}}{p \leq N}}a_p-\frac1{N}\sum_{n=1}^{N}\Lambda(n)a_n\Big|
& & \leq \Big|\frac{1}{\pi(N)}\sum_{\overset{p\textrm{~~prime}}{p \leq N}}a_p-\frac{1}{N}\sum_{n=1}^{N}\Lambda'(n)a_n\Big|\\
& &+ \Big|\frac{1}{N}\sum_{n=1}^{N}\Lambda'(n)a_n-\frac{1}{N}\sum_{n=1}^{N}\Lambda(n)a_n\Big|.
\end{eqnarray*}
Therefore
\begin{eqnarray*}
 &&\Big|\frac{1}{\pi(N)}\sum_{\overset{p\textrm{~~prime}}{p \leq N}}a_p-\frac1{N}\sum_{n=1}^{N}\Lambda(n)a_n\Big|\leq\\
 &&\Big|\frac{1}{\pi(N)}\sum_{\overset{p\textrm{~~prime}}{p \leq N}}a_p-\frac{1}{N}\sum_{n=1}^{N}\Lambda'(n)a_n\Big|+\frac{(\log N)^2}{2 \sqrt{N}\log N},
\end{eqnarray*}
by Proposition \ref{chebyshev}. Now, we need to estimate the first term in the RHS. For that, we observe that
\begin{eqnarray*}
 &&\Big|\frac{1}{\pi(N)}\sum_{\overset{p\textrm{~~prime}}{p \leq N}}a_p-\frac{1}{N}\sum_{n=1}^{N}\Lambda'(n)a_n\Big| \\
 &\leq&\Big|\frac{1}{N}\sum_{{\overset{p \leq N}{p \textrm{~prime}}}}\frac{N}{\pi(N)}a_p-\frac{1}{N}\sum_{{\overset{p \leq N}{p \textrm{~prime}}}}\log(p)a_p\Big|\\
 &\leq&\frac{1}{N}\sum_{{\overset{p \leq N}{p \textrm{~prime}}}} \Big|\frac{N}{\pi(N)}-\log(p)\Big|
 \end{eqnarray*}
But, by PNTR, we have
\begin{eqnarray}\label{PNTR}
 \Big|\frac{N}{\pi(N)}-\log(N)\Big| \leq \frac{C.N}{\pi(N).\log(N)},
\end{eqnarray}
where $C$ is an absolute positive constant. Moreover, for any $N \geq 2$, we have $\pi(N) \leq 6 \frac{N}{\log(N)}$. We can thus rewrite
\eqref{PNTR} as follows
\begin{eqnarray}\label{PNTR2}
 \Big|\frac{N}{\pi(N)}-\log(N)\Big| \leq 6C.
\end{eqnarray}
This combined with the triangle inequality gives
\begin{eqnarray*}
 &&
\frac{1}{N}\sum_{{\overset{\ell \leq N}{\ell \textrm{~prime}}}} \big|\frac{N}{\pi(N)}-\log(\ell)\big|\\
&\leq&
\frac{1}{N}\sum_{{\overset{\ell \leq N}{\ell \textrm{~prime}}}} \big|\frac{N}{\pi(N)}-\log(N)\big|
+ \frac{1}{N}\sum_{{\overset{\ell \leq N}{\ell \textrm{~prime}}}} \big(\log(N)-\log(\ell)\big),
\end{eqnarray*}
Whence
\begin{eqnarray*}
 && \Big|\frac{1}{\pi(N)}\sum_{\overset{p\textrm{~~prime}}{p \leq N}}a_p-\frac{1}{N}\sum_{n=1}^{N}\Lambda'(n)a_n\Big| \\
 &\leq&
 \frac{\pi(N)}{N}\Big|\frac{N}{\pi(N)}-\log(N)\Big|+\frac{\log(N).\pi(N)}{N}-\frac{\upsilon(N)}{N}\\
 &\leq& \frac{8C}{\log(N)}+\frac{6C^2}{(\log(N))^2}
 \end{eqnarray*}
 \end{proof}
From the proof of Lemma \ref{NH} the esimation of the sum over the prime is reduced to the estimation of the following quantity
$$ \frac{1}{N}\sum_{n=1}^{N}\Lambda(n)a_n
 \textrm{~~~or~~~}
\frac{1}{N}\sum_{n=1}^{N}\Lambda'(n)a_n.$$
Indeed, we have proved the following
$$ \Big|\displaystyle \frac{1}{\pi(N)}\sum_{n \leq N}\Lambda(n)a_n-\frac1{N}\sum_{n=1}^{N}\Lambda'(n)a_n\Big|
\leq \frac{(\log(N))^2}{2\sqrt{N}\log(2)},$$
and
$$ \Big|\displaystyle \frac{1}{\pi(N)}\sum_{\overset{p\textrm{~~prime}}{p \leq N}}a_p-\frac1{N}\sum_{n=1}^{N}\Lambda'(n)a_n\Big| \leq \frac{8C}{(\log(N))}+\frac{6C^2}{(\log(N))^2},
$$ for some absolute positive constant $C$.\\

\begin{proof}[\textbf{Proof of Proposition \ref{Mangoldt-nilsequence}.}] For the case $k=1$. The result holds for any dynamical system. Indeed,
the result follows from Lemma \ref{NH} combined with the main result
from \cite{Wierdl}. Let us assume from now that $k \geq 2$. Then,
as in Ford-Green-Konyagin-Tao's proof, it suffices to see that the following holds
$$\frac1{N^k}\sum_{\vec{n}\in [1,N]^k}\prod_{e \in C^*}\big(\Lambda'_{b_i,W}(\vec{n}.e)-1\big) \prod_{e \in C^*}f_e(T_e^{\vec{n}.e}x)), \eqno(NCSM)$$
where $b_i \in [1,W]$, $i=1,\cdots,\phi(W),$ coprime to $W$. But $\big(\Lambda'_{b_i,W}-1\big)$ is orthogonal to the nilsequences by Green-Tao result. Therefore the limit of the quantity $(NCSM)$ is zero.
\end{proof}
We are now able to prove our second main result.
\begin{proof}[\textbf{Proof of Theorem \ref{main-Mangoldt}.}]
We proceed as before by claiming that
$$\frac1{N^k}\sum_{\vec{n}\in [1,N]^k}\prod_{e \in C^*}\big(\Lambda'_{b_i,W}(\vec{n}.e)-1\big) \prod_{e \in C^*}f_e(T_e^{\vec{n}.e}x)) \tend{N}{+\infty}0.$$
If not then there exist $\delta>0$ a functions $f_e$, $e \in V$ and a Borel set $A$ with positive measure such that
for each $x\in A$
 \[
\limsup_{N \longrightarrow +\infty}\dfrac1{N^k}\sum_{{\boldsymbol{n}} \in [1,N]^k}\prod_{e \in C^*}\big(\Lambda'_{b_i,W}(\vec{n}.e)-1\big)f_{{\boldsymbol{e}}}
\big(T_{{\boldsymbol{e}}}^{{\boldsymbol{n}}.{\boldsymbol{e}}}x\big) \geq \delta.
\]
Whence, by the same arguments as in the proof of Theorem \ref{main}, we get that there exist a
nilmanifold $G/\Gamma$, a filtration $(G_p)$, a
polynomial sequence $g$ and a Lipschitz function $F$ such that
$$\frac1{N}\sum_{n=1}^{N} \big(\Lambda'_{b_i,W}(n)-1\big) f_e(T_e^{n}x) F(g(n)\Gamma) \gg_{k, \delta} 1.$$

\noindent We way further assume without lost of generalities that
$$\big|\big\|f_{e}(T_e^n(x))\big\|\big|_{U^{k+1}}=0,$$ and write, up to some small error, that
$$\frac1{N}\sum_{n=1}^{N} \big(\Lambda'_{b_i,W}(n)-1\big)f_e(T_e^{n}x) F(g(n)\Gamma) \gg_{k, \delta} 1.$$
Integrating with respect to $x$ and applying the spectral theorem, we see that
$$\Big\|\frac{1}{\phi(W)}\sum_{i=1}^{\phi(W)}\frac1{N}\sum_{n=1}^{N} \big(\Lambda'_{b_i,W}(n)-1\big) z^n F(g(n)\Gamma)\Big\|_{L^2(\sigma_{f_e})} \gg_{k, \delta} 1 .$$

\noindent Whence, by letting $W, N \longrightarrow + \infty$, we get
$$\limsup\Big\|\frac{1}{\phi(W)}\sum_{i=1}^{\phi(W)}\frac1{N}\sum_{n=1}^{N}  \big(\Lambda'_{b_i,W}(n)-1\big) z^n F(g(n)\Gamma)\Big\|_{L^2(\sigma_{f_e})} \gg_{k, \delta} 1,$$

\noindent which contradicts Green-Tao Theorem since the sequence
$$\Big(\frac{1}{\phi(W)}\sum_{i=1}^{\phi(W)}\frac1{N}\sum_{n=1}^{N}  \big(\Lambda'_{b_i,W}(n)-1\big) z^n F(g(n)\Gamma)\Big)$$
is bounded by the PNT, Lemma \ref{NH} and the triangle inequality, and it converges to zero. The proof of the theorem is complete.
\end{proof}
\begin{rem}
 Let us stress that according to Ford-Green-Konyagin-Tao's Theorem \cite{FGKtao}, if all the functions $f_e$ are constant (say equal $1$), then the limit is given by the following explicit expression
 $$\prod_{p}\beta_p,
\textrm{~~where~~}\beta_p=\frac1{p^d}\sum_{\vec{n} \in (\Z/p\Z)^d}\prod_{e \in C^*}\Lambda_{\Z/p\Z}(\vec{n}.e).$$
The function $\Lambda_{\Z/p\Z}$ is the local von Mangoldt function, that is, the $p$-periodic function defined by setting
$\Lambda_{\Z/p\Z}(b)=\frac{p}{p-1}$ when $b$ is coprime to $p$ and $\Lambda_{\Z/p\Z}(b)=0$ otherwise. 
\end{rem}

Our proof suggests the following questions.
\begin{que} Let $(X,\ca,T,\mu)$ a dynamical system, let $f \in L^1(X)$, do we have that there exists a Borel set $X'$ with full measure such that for any nilsequence
$(a_n)$, for any $x \in X'$, the following averages
$$1/\pi(N) \sum_{\overset{p \in \cp}{p < N}}a_p f(T^px),$$
converge? In other word, can one establish  Wiener-Wintner version of the prime ergodic theorem?
\end{que}

Let us notice further that the Gowers uniformity norm of $\bmu$ is small, that is, \linebreak $\||\bmu|\|_{U^k(N)}\tend{N}{+\infty}0$. 
This suggests the following generalization of Sarnak's conjecture.
\begin{que}
 Do we have that for any multiplicative function $(\bnu(n))$ with small \linebreak Gowers norms
 for any dynamical flow on a compact set $(X,T)$ with topological entropy zero, for any continuous function $f$, for all $x \in X$,
\[
\frac1{N}\sum_{n=1}^{N}\bnu(n)f(T^nx) \tend{N}{+\infty}0?
\]
\end{que}

\noindent By our assumption, it is obvious that for any nilsystem $(X,T),$  for any continuous function $f$, for any $x \in X$, we have
\[
\frac1{N}\sum_{n=1}^{N}\bnu(n)f(T^nx) \tend{N}{+\infty}0.
\]
Let's stress  that by the recent result of L. Matthiesen \cite{Lilian} there is a class of not necessarily bounded multiplicative
functions with a small Gowers norms.

\begin{thank}
The first author would like to express his heartfelt thanks to Benjamin Weiss
for the discussions on the subject. It is a great pleasure also for him to acknowledge the warm
hospitality of University of Science and Technology of China and Fudan University where this work has been done and revised.
\end{thank}


\end{document}